\documentclass{amsart}

\usepackage{amssymb}
\usepackage{amsfonts}
\usepackage{amsmath}
\usepackage{amsthm}
\usepackage{mathrsfs}
\usepackage{dsfont}
\usepackage{tikz}
\usepackage{theoremref}
\usepackage[colorlinks=true, urlcolor=blue, citecolor=blue, linkcolor=blue]{hyperref}

\usepackage{aliascnt}

\newtheorem{thm}{Theorem}[section]

\newaliascnt{lem}{thm}

\aliascntresetthe{lem}

\newtheorem{qst}[thm]{Question}%[qst]
\newaliascnt{cnj}{thm}  

\aliascntresetthe{cnj}

\newaliascnt{prp}{thm}  
\newtheorem{prp}[prp]{Proposition}
\aliascntresetthe{prp}

\newaliascnt{cor}{thm}  
\newtheorem{cor}[cor]{Corollary}
\aliascntresetthe{cor}

\theoremstyle{definition}
\newaliascnt{dfn}{thm}  

\aliascntresetthe{dfn}

\numberwithin{equation}{section}

\newcommand{\bbN}{\mathbb N}
\newcommand{\cU}{\mathcal U}
\newcommand{\cV}{\mathcal V}

\author{Tristan Bice and Ilijas Farah}
\address{
York University\\
Toronto\\
Canada
}
\email{Tristan.Bice@gmail.com}
\thanks{Partially supported by NSERC
and York University Post-Doctoral Fellowship}

\date{\today} 
\begin{document}

\title{Traces, Ultrapowers and the Pedersen-Petersen C*-Algebras}

\begin{abstract}
Our motivating question was whether all traces on a $\mathcal{U}$-ultrapower of a C*-algebra $A$, where $\mathcal{U}$ is a non-principal ultrafilter on $\mathbb{N}$, are necessarily $\mathcal{U}$-limits of traces on $A$.  We show that this is false so long as $A$ has infinitely many extremal traces, and even exhibit a $2^{2^{\aleph_0}}$ size family of such traces on the ultrapower.  For this to fail even when $A$ has finitely many traces implies that $A$ contains operators that can be expressed as sums of $n+1$ but not $n$ *-commutators, for arbitrarily large $n$.  We show that this happens for a direct sum of Pedersen-Petersen C*-algebras, and analyze some other interesting properties of these C*-algebras.
\end{abstract}

\maketitle

Let $A$ be a C*-algebra and let $\cU$ be an ultrafilter on $\bbN$. 
Every sequence $\tau_n$, for $n\in \bbN$, of traces on $A$ defines a trace
$\lim_{n\to \cU} \tau_n$ on the ultrapower $A^{\cU}$ (see \S\ref{S.Linear} for definitions). 
We say that such  traces on $A^{\cU}$ 
are \emph{trivial}. Note that $T(A)$ is in duality with $A^{\cU}$ if we extend $\tau\in T(A)$   to $A^{\cU}$
by $\tau((a_n))=\lim_{n\to \cU} \tau(a_n)$ (here $(a_n)\in l^\infty(A)$ is a representing sequence
of an element of $A^{\cU}$), and we have the following diagram.

\begin{tikzpicture}
 \matrix[row sep=0.5cm,column sep=1cm] {
\node  (A) {$A$};
& 
\node  (AU) {$A^{\cU}$};\\
\node  (TA) {$T(A)$};
& 
\node  (TAU) {$T(A)^{\cU}$};
& 
\node  (TAU1) {$T(A^{\cU})$};\\
};
\path (A) edge node[above] {$\subseteq$} (AU); 
\path (A) edge [<->]  (TA); 
\path (TA) edge node[above] {$\subseteq$}   (TAU); 
\path (TAU) edge node[above] {$\subseteq$}  (TAU1); 
\path (AU) edge [<->]  (TAU); 
\path (AU) edge [<->] (TAU1); 
\end{tikzpicture}

The space  $T(A)^{\cU}$ 
 of trivial traces on $A^{\cU}$ and the relative commutant $A^{\cU}\cap A'$ 
have played a pivotal  role in recent progress on classification program for nuclear C*-algebras
(see \cite{MatuiSato2012}). 
Motivated by this, during the BIRS workshop on 
Descriptive Set Theory and Operator Algebras in June 2012 Wilhelm Winter asked  
whether  every trace on the countable ultraproduct of unital,
separable, tracial C*-algebras comes from an ultraproduct of traces?
In an afternoon during the von Neumann conference in M\"unster (also in June 2012), 
the second author, Hiroki Matui, and Wilhelm Winter sketched a proof that 
the Continuum Hypothesis implies a negative answer if the C*-algebras in question are 
separable and have infinite-dimensional trace space.  
It turns out that 
the Continuum Hypothesis is not needed for this conclusion  (see \autoref{traces1} below). 
This  is a consequence of known results about ultrapowers of Banach spaces 
 and it was independently observed by Narutaka Ozawa. 
Now let us say that a trace $\tau$ on $A^{\cU}$ is \emph{locally trivial} if 
for every $a\in A^{\cU}$ there exists a sequence $\tau_n$ of traces on $A$ such that 
$\tau(a)=\lim_{n\to \cU} \tau_n(a)$ (this is weaker than stating that $\tau$ is 
in the weak*-closure of $T(A)$ inside $T(A^{\cU}$)). 
All nontrivial traces constructed in \autoref{traces1} are locally trivial. 
In \autoref{C1}  
we construct a C*-algebra $A$ and a trace in its ultrapower that is not even locally trivial.

\section{Linear Functionals on Normed Space Ultrapowers}
\label{S.Linear} 
Take a normed space $X$ where $\mathbb{F}$ denotes the scalar field, either the reals $\mathbb{R}$ or complex numbers $\mathbb{C}$, with unit ball $X^1=\{x:||x||\leq1\}$.  Let $\mathcal{U}$ be a non-principal ultrafilter on $\mathbb{N}$ and define
\begin{eqnarray*}
l^\infty(X) &=& \{(x_n)\subseteq X:\sup||x_n||<\infty\},\quad\textrm{ and }\\
c_\mathcal{U}(X) &=& \{(x_n)\in l^\infty(X):\lim_{n\rightarrow\mathcal{U}}||x_n||=0\}.
\end{eqnarray*}
As $c_\mathcal{U}(X)$ is a closed subspace of $l^\infty(X)$ (under pointwise addition and scalar multiplication), we can define the ($\mathcal{U}$-)\emph{ultrapower} $X^\mathcal{U}$ of $X$ by \[X^\mathcal{U}=l^\infty(X)/c_\mathcal{U}(X).\]  We denote the image of $(x_n)$ under the quotient map by $(x_n)^\mathcal{U}$.  Also note that, if $A$ is a C*-algebra, then $l^\infty(A)$ is too (under pointwise multiplication) containing $c_\mathcal{U}(X)$ as a (closed two-sided) ideal, making $A^\mathcal{U}$ a C*-algebra too.  Furthermore, the only property of the ultrafilter $\mathcal{U}$ that we use is that it is non-principal, and so most of the results will also apply to the asymptotic sequence space (or C*-algebra) $l^\infty(X)/c_0(X)$, where $c_0(X)=\{(x_n)\subseteq X:\lim_{n\rightarrow\infty}||x_n||=0\}$.

Note that there is a canonical embedding of $X^{*\mathcal{U}}$ in $X^{\mathcal{U}*}$, i.e. $X^{*\mathcal{U}}\subseteq X^{\mathcal{U}*}$.  Likewise, taking the $\mathcal{U}$-limit of a sequence of traces on a C*-algebra $A$ yields a trace on $A^\mathcal{U}$.  Our original motivating question was whether these are the only traces on $A^\mathcal{U}$.  This can be naturally decomposed into two questions, one of which is a special case of whether $X^{*\mathcal{U}}=X^{\mathcal{U}*}$.  This is because every self-adjoint tracial functional has a Jordan decomposition as a difference of traces (similarly to \cite{Pedersen1979} \S 3.2), and the space of self-adjoint tracial functionals on a C*-algebra $A$ is just the (real scalar) dual of $A^q=A_\mathrm{sa}/A_0$, where $A_0=\{\sum x_nx_n^*-x_n^*x_n\textrm{ (norm convergence)}:(x_n)\subseteq A\}$ (see \cite{CuntzPedersen1979}).  Thus we have $(A^\mathcal{U})_0\subseteq(A_0)^\mathcal{U}$ and hence \[A^{q*\mathcal{U}}\subseteq A^{q\mathcal{U}*}\subseteq A^{\mathcal{U}q*}.\]  So every trace on $A^\mathcal{U}$ is a $\mathcal{U}$-limit of traces on $A$ if and only if both these inclusions are actually equalities.

We now show that the first inclusion is strict whenever $A$ has infinitely many extremal traces (we show that the second inclusion can also be strict in the following section).  For normed spaces $X$, we have $X^{*\mathcal{U}}=X^{\mathcal{U}*}$ precisely when $X^\mathcal{U}$ is reflexive, by \cite{Heinrich1980} Proposition 7.1 (which occurs precisely when $X$ itself is superreflexive - see \cite{Heinrich1980} Corollary 7.2 - although we do not need this result here), and so all we have to do is show that $A^{q\mathcal{U}}$ is not reflexive.

\begin{thm}\label{traces1}
If $A$ is a C*-algebra with infinitely many extremal tracial states $(\tau_n)$ then there is a trace on $A^\mathcal{U}$ that is not a $\mathcal{U}$-limit of traces on $A$.
\end{thm}

\begin{proof}
For each extremal tracial state $\tau$ on $A$, we have a minimal central projection $p_\tau\in A^{**}$ (seen as the enveloping algebra of $A$ on a Hilbert space $H$), such that $(A|_{\mathcal{R}(p_\tau)})''$ is a factor with unique (normal) tracial state $\phi_\tau$ with $\tau(a)=\phi_\tau(a|_{\mathcal{R}(p_\tau)})$.  For distinct extremal tracial states $\tau$ and $\sigma$, we must therefore have $p_\tau p_\sigma=0$ and hence, for any $m\in\mathbb{N}$ and $r_1,\ldots,r_m\in l^1_\mathbb{R}$, \[\sum_{n\leq m}|r_n|=\sum_{n\leq m}r_n\tau_n(\sum_{k\leq m}\mathrm{sgn}(r_k)p_{\tau_k})\leq||\sum_{n\leq m}r_n\tau_n||\leq\sum_{n\leq m}|r_n|.\]  Thus $(\tau_n)$ is a basis of a copy of $l^1$ in $A^{q*}=(A/A_0)^*$.  As $l^1$ is not reflexive, neither is $A^{q*}$ (see \cite{Megginson1998} Theorem 1.11.16) or, for that matter, $A^q$ (see \cite{Megginson1998} Theorem 1.11.17).  As $A^{q\mathcal{U}}$ contains a copy of $A^q$, $A^{q\mathcal{U}}$ is not reflexive either (again see \cite{Megginson1998} Theorem 1.11.16) and hence $A^{q*\mathcal{U}}\subsetneqq A^{q\mathcal{U}*}\subseteq A^{\mathcal{U}q*}$, by \cite{Heinrich1980} Proposition 7.1.
\end{proof}

\begin{thm}\label{states1}
If $A$ is an infinite dimensional C*-algebra then there is a state on $A^\mathcal{U}$ that is not a $\mathcal{U}$-limit of states on $A$.
\end{thm}

\begin{proof}
Again, this will follow from \cite{Heinrich1980} Proposition 7.1 once we show that $A$ is not reflexive.  To see this, let $(\phi_n)$ be a sequence of pure states corresponding to an orthonormal sequence of vectors in the Hilbert space of the atomic representation of $A$.  Kadison's trasitivity theorem (actually von Neumann's bicommutant theorem is enough) now shows that this is a basis for a copy of $l^1$ in $A^*$ which, as in the previous proof, shows that $A$ is not reflexive.  Alternatively, note that every infinite C*-algebra contains a self-adjoint element $S$ with infinite spectrum.  Taking a sequence of functions $(f_n)\subseteq C(\sigma(S))$ of sup-norm $1$ with disjoint supports, we see that $f_n(S)$ is a basis of a copy of $c_0$ in $A$ which, as $c_0$ is not reflexive, means that $A$ is not either.
\end{proof}

While these proofs are nice and short, they do not give us any clue as to what these extra traces or states on $A^\mathcal{U}$ might look like or, indeed, how many of them there are.  However, we can answer both these questions with (the proof of) the following result.  Note we make repeated use of Goldstine's theorem, which says the image of $X^1$ under the canonical embedding in its double dual is weakly*-dense in $X^{**1}$ (which is a consequence of a standard separation theorem \textendash\, see \cite{Megginson1998} 2.6.26).\footnote{The operator algebraist may note the similarity between Goldstine's theorem and Kaplansky's density theorem.  Indeed, the latter could be seen as a corollary of the former, once the double dual of a C*-algebra is identified as its universal enveloping algebra (except that this argument is a bit circular, as the proof of this identification usually already involves Kaplansky's density theorem).}

\begin{thm}\label{mainthm}
For any infinite dimensional normed space $X$ such that $X^*$ contains an (isometric) copy of $l^1$, we have $X^{*\mathcal{U}}\subsetneqq X^{\mathcal{U}*}$.  In fact, $X^{\mathcal{U}*}\setminus(X^{*\mathcal{U}}\setminus\{0\})$ contains a copy of $l^1(2^{2^{\aleph_0}})$.
\end{thm}

\begin{proof}
Note that the following subsets of $\mathbb{N}\times\mathbb{N}$ have the finite intersection property
(i.e., the intersection of finitely many elements is always nonempty) 
 \begin{eqnarray*}
V_{\mathrm{diag}} &=& \{(m,n):n\leq m\},\\
V_f &=& \{(m,n):n \neq f(m)\},\textrm{ for }f\in\mathbb{N}^\mathbb{N},\textrm{ and}\\
V_U &=& U\times\mathbb{N},\textrm{ for }U\in\mathcal{U}.
\end{eqnarray*}
Thus we may let $\mathcal{V}$ be an ultrafilter on $\mathbb{N}\times\mathbb{N}$ containing all of them.  Given $(x_m)\in X^\mathbb{N}$, define $x^*((x_m))=\lim_{(m,n)\rightarrow\mathcal{V}}e^*_nx_m$, where $(e_n^*)\subseteq X^*$ denotes the canonical basis for $l^1$.  We will show that $x^*\notin X^{*\mathcal{U}}$.

As $V_U\in\mathcal{V}$, for all $U\in\mathcal{U}$, this yields a well-defined element of $X^{\mathcal{U}*}$.  Also, for each $m\in\mathbb{N}$, the linear functional $e_m^{**}$ on $\mathrm{span}\{e_1^*,\ldots,e_m^*\}$ defined by $e_m^{**}e_k^*=1$, for all $k\leq m$, has norm $1$ (and can be extended to an element of $X^{**}$ of norm $1$, by the Hahn-Banach theorem) and thus, by Goldstine's theorem, there exists $e_m\in X^1$ such that $e_k^*e_m\geq1-1/m$, for each $k\leq m$.  Thus $x^*((e_m)^\mathcal{U})=1$ and hence $1\leq||x^*||\leq\sup||e^*_n||=1$.

Say $x^*=(x_m^*)^\mathcal{U}$, for some $(x_m^*)\subseteq X^*$.  As $||x^*||=1$, we may renormalize if necessary to make $||x_m^*||=1$, for $m\in\mathbb{N}$.  We claim $\lim_{m\rightarrow\mathcal{U}}d(x_m^*,\mathrm{span}\{e_1^*,\ldots,e_m^*\})=0$.  For, if not, there exists $\epsilon>0$ and $U\in\mathcal{U}$ such that, for each $m\in U$, $d(x_m^*,\mathrm{span}\{e_1^*,\ldots,e_m^*\})\geq\epsilon$.  But then the functional $x_m^{**}$ on $\mathrm{span}\{e_1^*,\ldots,e_m^*,x_m^*\}$ with kernel $\{e_1^*,\ldots,e_m^*\}$ and $x_m^{**}x^*=\epsilon$ has norm at most $1$.  Applying Goldstine's theorem, we get $x_m\in X^1$ with $|\epsilon-x_m^*x_m|\leq\epsilon/3$ and $|e_k^*x_m|\leq\epsilon/3$, for all $k\leq m$, and hence $|x^*((x_m)^\mathcal{U})|\leq\epsilon/3$, a contradiction.

Thus we may assume $x_m^*=\sum_{k\leq m}r_{m,k}e_k^*$, for some $(r_{m,k})\subseteq\mathbb{F}^1$ satisfying $\sum_{k\leq m}|r_{m,k}|=1$, for all $m\in\mathbb{N}$.  We claim that $\lim_{m\rightarrow\mathcal{U}}\max_{k\leq m}|r_{m,k}|=0$.  Otherwise, there exists $U\in\mathcal{U}$ such that, for each $m\in U$, $\max_{k\leq m}|r_{m,k}|\geq\epsilon$.  But then the linear functional $y_m^{**}$ on $\mathrm{span}\{e_1^*,\ldots,e_m^*\}$ with $y_m^{**}e_{k_m}^*=1$, where $|r_{m,k}|$ attains its maximum at $k_m$, and $y^{**}_me_k^*=0$, for $k\neq k_m$, has norm $1$, and we again apply Goldstine's theorem to get $y_m\in X^1$ such that we have $|e_{m,k_m}^*y_m-1|\leq1/m$ and $|e_{m,k}^*y_m|\leq1/m$ for $k\neq k_m$.  As $V_{(k_m)}\in\mathcal{V}$, we have $x^*((y_m)^\mathcal{U})=0<\epsilon\leq\lim_{m\rightarrow\mathcal{U}}|x_m^*y_m|$, a contradiction.

So, finally, we must have $U\in\mathcal{U}$ such that, for all $m\in U$, $|r_{m,k}|\leq1/3$, for all $k\leq m$.  But then there is a linear functional $z_m^{**}$ on $\mathrm{span}\{e_1^*,\ldots,e_m^*\}$ with $z_m^{**}e_k^*=0$ or $|z_m^{**}e_k^*|=1$, for all $k\leq m$, and $1/3\leq z_m^{**}x_m^*\leq2/3$.  One more application of Goldstine's theorem gets us $(z_m)$ with $1/3-1/m\leq x_m^*z_m\leq2/3+1/m$ and $|e_k^*z_m|\notin(1/m,1-1/m)$.  Then $|x^*((z_m)^\mathcal{U})|\in\{0,1\}$ while $1/3\leq\lim_{m\rightarrow\mathcal{U}}x_m^*z_m\leq2/3$, another contradiction.

To prove the last statement in the theorem, let $(W_n)$ be a sequence of disjoint subsets of $V_\mathrm{diag}$ such that, for each $n\in\omega$, $|W_n\cap(\{m\}\times\mathbb{N})|\rightarrow\infty$ as $m\rightarrow\infty$ and, for each $X\subseteq\mathbb{N}$, set $W_X=\cup_{n\in X}W_n$.  For each ultrafilter $\mathcal{X}$ on $\mathbb{N}$, the subsets $(W_X)_{X\in\mathcal{X}}$ will have the finite intersection property with sets at the start, and may thus be extended to an ultrafilter $\mathcal{V}_\mathcal{X}$, yielding $x^*_\mathcal{X}\in X^{\mathcal{U}*}\setminus X^{*\mathcal{U}}$.  As there are $2^{2^{\aleph_0}}$ ultrafilters on $\mathbb{N}$ (this follows from the existence of an independent family of subsets of $\mathbb{N}$ of size $2^{\aleph_0}$ \textendash\, see \cite{Kunen1980} Chapter VIII Exercise A6), there are $2^{2^{\aleph_0}}$ such linear functionals and Goldstine's theorem again shows that the closure of their span is a copy of $l^1(2^{2^{\aleph_0}})$ contained in $X^{\mathcal{U}*}\setminus(X^{*\mathcal{U}}\setminus\{0\})$.
\end{proof}

Note that the statement of the theorem above, while quite natural, is not the strongest that could be made from the given proof.  For one thing, we do not actually require $X^*$ to contain a copy of the entirety of $l^1$, we only really require it to contain a copy of $l^1(m)$, for all $m\in\mathbb{N}$.  In fact the same result would hold if we had an ultraproduct of a sequence of normed spaces $(X_m)$, rather than an ultrapower of a single $X$, so long as each $X_m^*$ contained a copy of $l^1(m)$.  Also, it would even be sufficient for each $X_m^*$ to contain an approximately isometric copy of $l^1(m)$, so long as the approximations get arbitrarily close as $n$ increases (in the terminology of \cite{HensonMoore1974}, it suffices that $l^1$ is finitely $1$-representable in $X$, or, in the terminology of \cite{Day1973}, it suffices that $X$ mimics $l^1$).

A slightly more interesting strengthening can be obtained if we assume that the basis of the copy of $l^1(m)$ can always be extended to a basis of (a copy of) $l^1(m+1)$ such that $l^1(m)+\mathbb{F}x_m^*=l^1(m+1)$, given any $x_m^*\in X^*\backslash l^1(m)$ (which is satisfied by $l^1$ itself, for example).  For then we can actually show that, for the $x^*\in X^{\mathcal{U}*}$ in the above proof,
\begin{equation}\label{dx*}
d(x^*,X^{*\mathcal{U}})=1.
\end{equation}
To see this, keep the first paragraph of the above proof and replace the rest with the following argument.  Write $x_m^*=s_mf_m^*+\sum_{k\leq m}r_{m,k}e_k^*$, where $f_m^*$ is the extra basis vector of $l^1(m+1)$, and $\sum_{k\leq m}|r_{m,k}|=||(x_m^*)^\mathcal{U}||=M$, say, for all $m\in\mathbb{N}$.  We fix $n\in\mathbb{N}$ and, for each $m\geq n$, let $S_m$ be an $n$-element subset of $\{1,\ldots,m\}$ on which $|r_{m,k}|$ obtains its $n$ largest values, for $k\leq m$.  Now partition $\{1,\ldots,m\}\backslash S_m$ into sets $(S_{m,j})_{j\leq n}$ such that $\sum_{k\in S_{m,j}}|r_{m,k}|\leq 2M/n$, for each $j\leq n$.  We then define $z_{m,j}^{**}\in X^{**1}$ by $z_{m,j}^{**}f_m^*=0$, $z_{m,j}^{**}e_k^*=1$, if $k\in S_{m,j}$, and $z_{m,j}^{**}e_k^*=0$ otherwise.  Applying Goldstine's theorem in the standard way gives us $(z_{m,j})\subseteq X^1$ such that $x^*((z_{m,j})^\mathcal{U})=1$, for some $j\leq n$, even though $\lim_{m\rightarrow\mathcal{U}}|x_m^*z_{m,j}|\leq2M/n$.  As $n$ was arbitrary, we have $||x^*-(x_m^*)^\mathcal{U}||\geq1$ which, as $(x_m^*)\in X^{*\mathcal{U}}$ was arbitrary, yields (\ref{dx*}).

%Another thing worth pointing out is that, if $|X^*|=2^{\aleph_0}$ (in particular, if $X$ is separable), then $|X^{*\mathcal{U}}|=2^{\aleph_0}$, and we would not even have to bother showing that the $x^*$ constructed above is not in $X^{*\mathcal{U}}$, as done in the first few paragraphs of the proof, because we could jump to the last part of the proof showing that there are $2^{2^{\aleph_0}}$ distinct linear functionals in $X^{\mathcal{U}*}$ of this form, and so at least some of them (in fact most of them) would have to lie outside $X^{*\mathcal{U}}$.  Indeed, in the original discussion between the second named author, Hiroki Matui and Wilhelm Winter, the idea arose to do a recursive construction of length $\aleph_1$ of distinct elements of $X^{*\mathcal{U}}$ (actually traces on a C*-algebra, rather than arbitrary linear functionals on a normed space, were being considered in this discussion) to get out $2^{\aleph_1}$ elements of $X^{\mathcal{U}*}$ in the end (in a somewhat similar manner to the construction of an outer automorphism of the Calkin algebra given in \cite{PhillipsWeaver2007}) and then use the continuum hypothesis and the separability of $X$ to conclude that $X^{*\mathcal{U}}\subsetneqq X^{\mathcal{U}*}$.  In trying to work out the details for this idea, we came up with the simpler and more general argument given in the proof of \autoref{mainthm}.

If we happen to know that $X$ contains a copy of $c_0$, then Hahn-Banach yields a copy of $l^1$ in $X^*$ and so \autoref{mainthm} applies.  However, the proof is even easier then, as we can use this copy of $c_0$ to avoid all but the second application of Goldstine's theorem in the proof of \autoref{mainthm}.  In fact, there is another shortcut if we just want to show that $X^*$ contains a copy of $l^1(2^{2^{\aleph_0}})$.  For if we let $e'_{m,n}=0$, for $m<n$, and $e'_{m,n}=e_n$ otherwise then $((e'_{m,n})^\mathcal{U})_{n\in\mathbb{N}}$ is the canonical basis for a copy of $l^\infty$ in $X^\mathcal{U}$.  Then the closed linear span of the linear functionals on $l^\infty$ defined by ultrafilters is the required copy of $l^1(2^{2^{\aleph_0}})$ in $X^{\mathcal{U}*}$.  Another indication of the large size of $(c_0)^\mathcal{U}$ is the fact that it contains a copy of $c_0(2^{\aleph_0})$.  In fact the same argument yields the following.

\begin{prp}\label{P.lp}
For $1\leq p<\infty$, $(l^p)^{\mathcal{U}}$ contains a copy of $l^p(2^{\aleph_0})$.
\end{prp}

\begin{proof}
First use a standard argument to construct a continuum size almost disjoint family $\mathcal{A}\subseteq\mathbb{N}^\mathbb{N}$.  Specifically, let $f:\{0,1\}^{<\omega}\rightarrow\mathbb{N}$ be one-to-one and, for any $g\in\{0,1\}^\omega$, define $a_g(n)=f(g\upharpoonright n)$.  If $g(n)\neq h(n)$ then $a_g(m)\neq a_h(m)$, for all $m>n$, and so $(a_g)_{g\in\{0,1\}^\omega}$ is the required family.  Now simply note that if $(e_n)$ is the canonical basis for $l^p$ and, for each $a\in\mathbb{N}^\mathbb{N}$, we let $e_a=(e_{h(n)})^\mathcal{U}$ then $(e_a)_{a\in\mathcal{A}}$ will be the canonical basis for a copy of $l^p(2^{\aleph_0})$ in $(l^p)^{\mathcal{U}}$.
\end{proof}

Of course the above is a special case of the  fact  that if a model contains an 
infinite indiscernible set then its ultrapower contains an indiscernible set of cardinality $2^{\aleph_0}$ (cf. proof of Theorem~5.6 (1) in \cite{FarahHartSherman2012}).

As $l^1(2^{\aleph_0})^*=l^\infty(2^{\aleph_0})$ \autoref{P.lp} implies 
 that, for any normed space $X$ containing a copy of $l^1$, $X^{\mathcal{U}*}$ will contain $2^{2^{\aleph_0}}$ elements a distance of at least $1$ away from each other.  In particular, if we also had $|X^*|=2^{\aleph_0}$ then we would again have $X^{*\mathcal{U}}\subsetneqq X^{\mathcal{U}*}$ (although again, such an $X$ would not be reflexive and so this also follows immediately from \cite{Heinrich1980} Proposition 7.1).

Getting back to traces and states C*-algebras, we see that the proofs of \autoref{traces1} and \autoref{states1}, together with \autoref{mainthm}, immediately yield the following corollaries.

\begin{cor}
If $A$ is a C*-algebra with infinitely many extremal tracial states $(\tau_n)$ then there are $2^{2^{\aleph_0}}$ traces on $A^\mathcal{U}$ that are not $\mathcal{U}$-limits of traces on $A$.
\end{cor}

\begin{cor}
If $A$ is an infinite dimensional C*-algebra then there are $2^{2^{\aleph_0}}$ states on $A^\mathcal{U}$ that are not $\mathcal{U}$-limits of states on $A$.
\end{cor}

The following may be worth recording.

\begin{prp} 
With $A$ and $\cU$ as above, every $\cU$-limit 
of extremal traces on $A$ is an extremal trace on $A^{\cU}$ and
every $\cU$-limit 
of extremal states on $A$ is an extremal state on $A^{\cU}$. 
\end{prp} 

\begin{proof} 
  For states, this follows from the continuous model theory version of \L os's theorem (see \cite{FarahHartSherman2012} Proposition 4.3), because a state $\phi$ is extremal if and only if \[\sup_{Q\in A^1_+}(\inf_{P\in A^1_+,\phi(P)=1}||PQP||-\phi(Q))=0\] (by \cite{Bice2011} Theorem 2.4, because this implies that the subset of $A^1_+$ on which $\phi$ is $1$ is \emph{maximal} norm centred).  For traces, this again follows from \L os's theorem and the fact that a trace is extremal if and only if it gives rise to a factor representation in the GNS construction, and hence if and only if it satisfies \[\sup_{Q\in A^1_+}(\sqrt{\tau(Q^2)-\tau(Q)^2}-\sup_{P\in A^1_+}\sqrt{\tau(-(PQ-QP)^2)})=0\] (see the proof of \cite{FarahHartSherman2012} Proposition 3.4 (1) and replace the non-trivial central projection $p$ in the weak/strong closure of $\pi_\tau[A]$ with $q\in A^1_+$ such that $||(p-\pi_\tau(q))\xi_\tau||$ is sufficiently small).
\end{proof} 

Now consider the ultrafilter 
$\cV$ constructed from $\cU$ in the proof of \autoref{mainthm}. 
As shown in the course of  this proof, $\cV$-limits of traces of $A$ are traces of $A^{\cU}$, 
and similarly $\cV$-limits of states of $A$ are states of $A^{\cU}$. 
We can thus ask whether  $\mathcal{V}$-limits of extremal states or traces are themselves extremal.  
Unlike the case of $\cU$-limits, the answer depends on  
the C*-algebra we are taking the ultrapower of.  If $A$ is $c_0$ (or its unitization) then it does indeed follow that a $\mathcal{V}$-limit of pure states/traces is extremal, as the projections in $A^\mathcal{U}$ corresponding to the elements of $\mathcal{V}$ form a maximal (norm) centred subset (and, as $c_0$ has real rank zero, it suffices to consider subsets of projections rather than positive contractions \textendash\, see \cite{Bice2011} \S4).  However, if $A$ is the C*-algebra of compact operators $\mathcal{K}(H)$ (or its unitization) on a separable infinite dimensional Hilbert space $H$, then verifying that $\mathcal{V}$-limits of pure states on $A$ are pure states on $A^\mathcal{U}$ already amounts to verifying something very similar to the long standing Kadison-Singer conjecture (see e.g.,  \cite{CasazzaTremain2006}, \cite{Weaver2004}). 
  For note that if we made $\inf_{j\in X_k}||P_ke_j||=1$ in (*) below (so $P_k=P_{X_k}$), as well as making $||PP_k||\leq1-\delta$, for some $\delta>0$ depending only on $\epsilon$, we would have a statement equivalent to the Kadison-Singer conjecture, for then it would imply \cite{Bice2011} Theorem 5.1 (v) and follow from \cite{Bice2011} Theorem 5.1 (iv) (although it remains to be seen whether (*) really is a weakening, for it could well be equivalent to the Kadison-Singer conjecture or, on the other hand, perhaps easily shown to be true without verifying the Kadison-Singer conjecture).

\begin{thm}
Let $(\phi_m)$ be pure states on $\mathcal{K}(H)$ corresponding to an orthonormal sequence $(e_m)$ of vectors on $H$.  If the Kadison-Singer conjecture holds then their $\mathcal{V}$-limit $\phi$ on $\mathcal{K}(H)^\mathcal{U}$ is pure, for all $\mathcal{V}$ satisfying the hypothesis of \autoref{mainthm}.  On the other hand, all such $\mathcal{V}$-limits are pure then the following weakening of the Kadison-Singer conjecture holds
\begin{itemize}
\item[(*)] For all $\epsilon>0$ there exists $n\in\mathbb{N}$ such that for all $P\in\mathcal{P}(\mathcal{K}(H))$ with $\sup||Pe_k||\leq1-\epsilon$ and all $m\in\mathbb{N}$ there exists a partition $X_1,\ldots,X_n$ of $m$ and $P_1,\ldots,P_n\in\mathcal{P}(\mathcal{K}(H))$ such that $\inf_{j\in X_k}||P_ke_j||>||PP_k||$, for all $k\leq m$.
\end{itemize}
\end{thm}

\begin{proof}
Assuming the Kadison-Singer conjecture holds, we show that \[\mathcal{P}(\phi)=\{(P_m)^\mathcal{U}:(P_m)\subseteq\mathcal{P}(\mathcal{K}(H))\textrm{ and }\lim_{(m,n)\rightarrow\mathcal{V}}\phi_n(P_m)=1\}\] is maximal norm centred (in fact maximal norm linked).  So take some $(Q_m)\subseteq\mathcal{K}(H)$ such that $\lim_{m\rightarrow\mathcal{U}}||P_mQ_m||=1$ for all $(P_m)\subseteq\mathcal{P}(\mathcal{K}(H))$ with $(P_m)^\mathcal{U}\in\mathcal{P}(\phi)$.  By \cite{Bice2011} Theorem 5.1 (iv),
\begin{itemize}
\item[(**)] for any $\epsilon>0$ and $P\in\mathcal{P}(\mathcal{B}(H))$, we have $X_1,\ldots,X_n\subseteq\mathbb{N}$ such that $\bigcup X_k=\mathbb{N}$ and $||PP_{X_k}||^2+||P^\perp P_{X_k}||^2<1+\epsilon$, for all $k\leq n$
\end{itemize}
Note here that $n$ may be assumed to depend only on $\epsilon$, for if we had a sequence $(P'_n)\subseteq\mathcal{P}(\mathcal{B}(H))$ such that $||P'_nP_{X_k}||^2+||P'^\perp_nP_{X_k}||^2\geq 1+\epsilon$ for some $k\leq n$ whenever $X_1,\ldots,X_n$ is a partition of $\mathbb{N}$, then $\prod P'_n\in\prod_\mathbb{N}\mathcal{P}(\mathcal{B}(H))\subseteq\mathcal{P}(\mathcal{B}(\bigoplus_\mathbb{N}H))\approx\mathcal{P}(\mathcal{B}(H))$ would witness the failure of (**).  Thus we have $n$ such that, for each $m$, we have a partition $X_{m,1},\ldots,X_{m,n}$ of $\mathbb{N}$ with $||Q_mP_{X_{m,k}}||^2+||Q_m^\perp P_{X_{m,k}}||^2<1+\epsilon$, for all $k\leq n$.  Now $\bigcup_m\{m\}\times X_{m,k}\in\mathcal{V}$, for some $k\leq n$, and thus $(P_{X_{m,k}\cap\{1,\ldots,m\}})^\mathcal{U}\in\mathcal{P}(\phi)$.  By hypothesis, $\lim_{m\rightarrow\mathcal{U}}||P_{X_{m,k}\cap\{1,\ldots,m\}}Q_m||=1$ and hence $\lim_{m\rightarrow\mathcal{U}}||P_{X_{m,k}\cap\{1,\ldots,m\}}Q_m^\perp||<\sqrt{\epsilon}$, which implies that $\phi_j(Q_m)>1-\epsilon$, for all $m\in U$, for some $U\in\mathcal{U}$, and for all $j\leq m$ with $j\in X_{m,k}$.  Thus $\lim_{(m,j)\rightarrow\mathcal{V}}\phi_j(Q_m)\geq1-\epsilon$.  As $\epsilon>0$ was arbitrary, we must have $\lim_{(m,j)\rightarrow\mathcal{V}}\phi_j(Q_m)=1$, i.e. $(Q_m)^\mathcal{U}\in\mathcal{P}(\phi)$, as required.

On the other hand, if (*) fails then
\begin{itemize}
\item[$\neg$(*)] There exists $\epsilon>0$ such that, for all $n\in\mathbb{N}$, there exists $Q\in\mathcal{P}(\mathcal{K}(H))$ with $\sup||Qe_k||\leq1-\epsilon$ and $m_n\in\mathbb{N}$ such that whenever $X_1,\ldots,X_n$ is a partition of $m_n$ and $P_1,\ldots,P_n\in\mathcal{P}(\mathcal{K}(H))$ there exists $k\leq n$ with $\inf_{j\in X_k}||P_ke_j||\leq||QP_k||$.
\end{itemize}
Take such an $\epsilon>0$ and, for each $m$ with $m_n\leq m<m_{n+1}$, let $Q_m$ be this $Q$.  Now the collection of all $S\subseteq\mathbb{N}\times\mathbb{N}$ such that there exists $U\in\mathcal{U}$ and $(P_m)\subseteq\mathcal{P}(\mathcal{K}(H))$ with $\inf_{\{j:(m,j)\in S\textrm{ and }j\leq m\}}||P_me_j||>||Q_mP_m||$, for all $m\in U$, forms a proper ideal.  Thus we may let $\mathcal{V}$ be an ultrafilter on $\mathbb{N}\times\mathbb{N}$ containing no such $S$.  As $\sup_k||Q_me_k||\leq1-\epsilon$, for all $m$, it follows that $\phi((Q_m)^\mathcal{U})<1$, where $\phi$ is the $\mathcal{V}$-limit of the pure states determined by $(e_k)$.  But if we take some $(P_m)\subseteq\mathcal{P}(\mathcal{K}(H))$ with $\phi((P_m)^\mathcal{U})=1$ then, for any $\delta>0$, we have $V\in\mathcal{V}$ with $\inf_{(m,k)\in V}||P_me_k||\geq1-\delta$.  But $V$ is not one of the sets $S$ so that means that, for any $U\in\mathcal{U}$, there exists $m\in U$ such that $||Q_mP_m||\geq\inf_{\{j:(m,j)\in V\textrm{ and }j\leq m\}}||P_me_j||\geq1-\delta$.  As $\delta>0$ was arbitrary, $||(Q_m)^\mathcal{U}(P_m)^\mathcal{U}||=1$.  This shows that the collection of projections on which $\phi$ is $1$ is not maximal norm centred and hence $\phi$ is not pure.
\end{proof}

\section{The Pedersen-Petersen C*-Algebras}

To see that we can indeed have $(A^\mathcal{U})_0\subsetneqq(A_0)^\mathcal{U}$, we look at the algebras considered in \cite{PedersenPetersen1970}.  Specifically, given $n\in\mathbb{N}$, let $\mathbb{C}P^n$ denote $n$-dimensional complex projective space (i.e. the one dimensional subspaces of $\mathbb{C}^{n+1}$ with their natural topology), and let $\mathfrak{P}_n$ be the $C^*$-algebra of continuous sections of the following vector bundle
\[B_n=\{
(x,\begin{bmatrix}
a & \mathbf{b}\\
\mathbf{c} & d
\end{bmatrix})
:x\in\mathbb{C}P^n; a,d\in\mathbb{C};\mathbf{b},\overline{\mathbf{c}}\in x\}\]
(where $\overline{\mathbf{c}}=\overline{(c_1,\ldots,c_{n+1})}=(\overline{c}_1,\ldots,\overline{c}_{n+1})$), with multiplication and ${}^*$ defined pointwise by
\[\begin{bmatrix}
a & \mathbf{b}\\
\mathbf{c} & d
\end{bmatrix}
\begin{bmatrix}
a' & \mathbf{b}'\\
\mathbf{c}' & d'
\end{bmatrix}
=
\begin{bmatrix}
aa'+\mathbf{b}\cdot\mathbf{c}' & a\mathbf{b}'+d\mathbf{b}\\
a'\mathbf{c}+d\mathbf{c}' & dd'+\mathbf{b}'\cdot\mathbf{c}
\end{bmatrix}\qquad\textrm{and}\qquad\begin{bmatrix}
a & \mathbf{b}\\
\mathbf{c} & d
\end{bmatrix}^*
=
\begin{bmatrix}
\overline{a} & \overline{\mathbf{c}}\\
\overline{\mathbf{b}} & \overline{d}
\end{bmatrix}.\]

In \cite{PedersenPetersen1970} Lemma 3.5, they showed that the constant sections $P_n=\begin{bmatrix}1&\mathbf{0}\\ \mathbf{0}&0\end{bmatrix}$ and $Q_n=\begin{bmatrix}0&\mathbf{0}\\ \mathbf{0}&1\end{bmatrix}$ in $\mathfrak{P}_n$ require more than $n$ operators to witness their Cuntz-Pedersen equivalence.  We use essentially the same idea to prove the stronger statement that $P_n-Q_n$ requires more than $n$ *-commutators to witness its membership of $(\mathfrak{P}_n)_0$ (in fact, something even slightly stronger).

\begin{thm}\label{*com}
If $T\in\mathfrak{P}_n$ is a sum of $n$ *-commutators then $||T-(P_n-Q_n)||=1$
\end{thm}

\begin{proof}
Note that
\[\begin{bmatrix}
a & \mathbf{b}\\
\mathbf{c} & d
\end{bmatrix}
\begin{bmatrix}
\overline{a} & \overline{\mathbf{c}}\\
\overline{\mathbf{b}} & \overline{d}
\end{bmatrix}
-
\begin{bmatrix}
\overline{a} & \overline{\mathbf{c}}\\
\overline{\mathbf{b}} & \overline{d}
\end{bmatrix}
\begin{bmatrix}
a & \mathbf{b}\\
\mathbf{c} & d
\end{bmatrix}
=
\begin{bmatrix}
|\mathbf{b}|^2-|\mathbf{c}|^2 & \ldots\\
\ldots & |\mathbf{c}|^2-|\mathbf{b}|^2
\end{bmatrix}.\]
Take $m_1,\ldots,m_k\in\mathfrak{P}_n$, i.e. $m_i(x)=\begin{bmatrix}
a_i(x) & \mathbf{b}_i(x)\\
\mathbf{c}_i(x) & d_i(x)
\end{bmatrix}$
and assume \[||\begin{bmatrix}
1 & \mathbf{0}\\
\mathbf{0} & -1
\end{bmatrix}-\sum(m_im_i^*-m_i^*m_i)||<1.\]
Identify the $(2n+1)$-sphere $S^{2n+1}$ with the set of norm $1$ vectors in $\mathbb{C}^{n+1}$ and for each $i$ define $f_i:S^{2n+1}\rightarrow\mathbb{C}$ by $f_i(\mathbf{e})\mathbf{e}=\mathbf{b}_i(\mathbb{C}\mathbf{e})-\overline{\mathbf{c}}_i(\mathbb{C}\mathbf{e})$.  Then $f=(f_1,\ldots,f_k)$ defines a map from $S^{2n+1}$ to a subset of $\mathbb{C}^k$ with $f(-\mathbf{e})=-f(\mathbf{e})$ avoiding $0$.  Thus $f/||f||$ is a continuous map from $S^{2n+1}$ to $S^{2k-1}$ taking antipodal points to antipodal points so, by the Borsuk-Ulam theorem, $k>n$.
\end{proof}

\begin{cor}\label{P_n}\label{C1} 
For $A=\bigoplus_n\mathfrak{P}_n$, we have $(A^\mathcal{U})_0\subsetneqq(A_0)^\mathcal{U}$.
In particular, there exists a trace $\tau$ on $A^{\cU}$ which is not locally trivial: 
for some $a\in A^{\cU}$ we have $\tau(a)\neq\lim_{n\to \cU}\tau_n(a)$ for any sequence $\tau_n$
in $T(A)$. 
\end{cor}

\begin{qst}
Can \autoref{P_n} be proved for some C*-algebra $A$ with a unique trace?  What about no traces, or finitely many extremal traces?  If 
so, can we also ensure that~$A$ has other regularity properties
such as separability, simplicity, or nuclearity? 
\end{qst}

By \cite[Theorem~8]{Ozawa2013} an exact, $\mathcal{Z}$-stable C*-algebra $A$ cannot satisfy the 
conclusion of \autoref{P_n}. Therefore any example of a simple 
nuclear C*-algebra with this property would
have to be nonclassifiable (for terminology see e.g., \cite{ElliotToms2008}).  

%To answer this question, some continuous model theory might be useful.  For these conditions on the number of traces can each be naturally expressed as omitting a certain partial type.  So there is a hope that one might be able to start with a C*-algebra with many traces and then apply an omitting types theorem to obtain the required C*-algebra.

We can  use the Borsuk-Ulam theorem in a similar manner to the proof of \autoref{*com} to get another interesting fact about $\mathfrak{P}_n$, for $n\geq2$.  Specifically, take $m\in\mathfrak{P}_n$ as above and let $b(\mathbf{e})\mathbf{e}=\mathbf{b}(\mathbb{C}\mathbf{e})$ and $c(\mathbf{e})\mathbf{e}=\mathbf{c}(\mathbb{C}\mathbf{e})$.  Then $f=(b,c)$ defines a map from $S^{2n+1}$ to a subset of $\mathbb{C}^2$ with $f(-\mathbf{e})=-f(\mathbf{e})$.  If this map avoided $0$ then $f/||f||$ would be a continuous map from $S^{2n+1}$ to $S^3$ taking antipodal points to antipodal points, contradicting the Borsuk-Ulam theorem as above when $n\geq2$.  So $m(x)$ has to be diagonal for some $x\in\mathbb{C}P^n$ and, in particular, there can not be any partial isometry $U\in\mathfrak{P}_n$ such that $U^2=0$.  This means $\mathfrak{P}_n$ can not be isomorphic to any C*-algebra $C(X,M_2)$ of continuous functions from some topological space $X$ to $M_2$

To prove the same fact about $\mathfrak{P}_1$ requires more work, using the cohomology theory explained in \cite{RaeburnWilliams1998}.  On the plus side, this theory allows us to show that $\mathfrak{P}_1$ is not even isomorphic to a corner of such a C*-algebra.  In fact, as $\mathfrak{P}_1$ is 2-homogeneous and its spectrum $\mathbb{C}P^1$ has (covering) dimension 2, this may well be the lowest dimensional example of a homogeneous C*-algebra known to be so twisted that it has this property.

\begin{thm}\label{P1nontriv}
$\mathfrak{P}_1$ is not isomorphic to any C*-algebra of the form $PC(X,M_n)P$, where $X$ is a topological space and $P$ is a (everywhere rank 2) projection in $M_n$.
\end{thm}

\begin{proof}
We first define some appropriate local trivializations of the bundle $B_1$.  Consider the open cover $U_0,\ldots,U_3$ of $\mathbb{C}P^1$ where $U_0=\{\mathbb{C}(1,z):|z|<1\}$ and $U_k=\mathbb{C}P^1\setminus U'_k$, for $k=1,2,3$, where
\begin{eqnarray*}
U'_1 &=& \{\mathbb{C}(1,re^{i\theta}):0\leq r\leq1\textrm{ and }-\pi/2\leq\theta\leq\pi/2\},\\
U'_2 &=& \{\mathbb{C}(1,re^{i\theta}):0\leq r\leq1\textrm{ and }-\pi\leq\theta\leq-\pi/2\},\textrm{ and}\\
U'_3 &=& \{\mathbb{C}(1,re^{i\theta}):0\leq r\leq1\textrm{ and }\pi/2\leq\theta\leq\pi\}.
\end{eqnarray*}
Given $z\in\mathbb{C}$, let
\begin{eqnarray*}
\mathbf{z}'_0 &=& (1,z),\\
\mathbf{z}'_1 &=& (-iz,1-z),\\
\mathbf{z}'_2 &=& (z,1-z),\textrm{ and}\\
\mathbf{z}'_3 &=& (iz,1-z).
\end{eqnarray*}
Note that, for $k=0,\ldots,3$, the map $z\mapsto\mathbb{C}\mathbf{z}'_k$ is a one-to-one from $\mathbb{C}$ to $\mathbb{C}P^1$ minus the points $\mathbb{C}(0,1),\mathbb{C}(1,-i),\mathbb{C}(1,-1),$ and $\mathbb{C}(1,i)$ respectively, which are not in $U_0,\ldots,U_3$ respectively.  So we may let $\mathbf{z}_k=||\mathbf{z}'_k||^{-1}\mathbf{z}'_k$ and
\[h_k(\mathbb{C}\mathbf{z}_k,\begin{bmatrix}a&b\\ c&d\end{bmatrix})=(\mathbb{C}\mathbf{z}_k,\begin{bmatrix}a&b\mathbf{z}_k\\ c\overline{\mathbf{z}}_k&d\end{bmatrix}),\textrm{ for }\mathbb{C}\mathbf{z}_k\in U_k.\]

Let $\sqrt{re^{i\theta}}=\sqrt{r}e^{i\theta/2}$, for $r\geq0$ and $-\pi<\theta<\pi$, i.e. we are now specifying that $\sqrt{}$ always denotes a certain continuous branch of the square-root function on $\mathbb{C}$ minus the negative reals.  For $j,k=0,\ldots,3$ and $j>k$, define $z'_{jk}$ on $U_{jk}=U_j\cap U_k$ by
\begin{eqnarray*}
z'_{10}(\mathbb{C}(1,z)) &=& \sqrt{z+i},\\
z'_{20}(\mathbb{C}(1,z)) &=& \sqrt{z+1},\\
z'_{30}(\mathbb{C}(1,z)) &=& \sqrt{z-i},\\
z'_{21}(\mathbb{C}(1,z)) &=& \sqrt{(z+1)/(z+i)},\\
z'_{31}(\mathbb{C}(1,z)) &=& \sqrt{(z-i)/(z+i)},\\
z'_{32}(\mathbb{C}(1,z)) &=& \sqrt{(z-i)/(z+1)},\\
\end{eqnarray*}
(with $z'_{jk}(\mathbb{C}(0,1))=1$ for $j,k\neq0$) and let $z_{jk}=z'_{jk}/|z'_{jk}|$.  Then, for all $x$ in the appropriate domains, we have
\[h_j^{-1}\circ h_k(x,m)=(x,\mathrm{Ad}(\begin{bmatrix}z_{jk}(x)&0\\ 0&\overline{z}_{jk}(x)\end{bmatrix})(m)).\]
For each $j,k,l=0,1,2,3$ with $j>k>l$ there exists a unique $\delta_{jkl}\in\{-1,1\}$ such that
\[z_{jl}(x)=\delta_{jkl}z_{jk}(x)z_{kl}(x),\textrm{ for all }x\in U_{jkl}=U_j\cap U_k\cap U_l.\]
Direct calculation shows that (for our given choice of the branch of $\sqrt{}$) $\delta_{210}=\delta_{320}=\delta_{321}=1$ while $\delta_{310}=-1$.  This $(\delta_{jkl})$ is a (2-)cocycle (for the trivial reason that $U_0\cap\ldots\cap U_3=\emptyset$) but not a (2-)coboundary, i.e. it represents a (in fact the) non-identity element of the second \v{C}ech cohomology group $H^2((U_k),\mathcal{S})$, where $\mathcal{S}$ is the sheaf of germs of continuous (and hence constant, as all intersections of the $(U_k)$ are connected) $\{-1,1\}$-valued functions on $\mathbb{C}P^1$ (probably the easiest way to see this is to identify this group with the second simplicial cohomology group with coefficients in $\mathbb{Z}_2\approx\{1,-1\}$ of the boundary of the 3-simplex \textendash\, see \cite{Munkres1984} \S73).  This means that $B_1$ is not isomorphic to any $\mathrm{Aut}(M_2)$-bundle where the transition functions $\mathrm{Ad}(u_{jk})$ come from a 1-cocycle $(u_{jk})$, by \cite{RaeburnWilliams1998} Lemma 4.81 (note we are not using the exact sequence $1\rightarrow\mathbb{T}\rightarrow U(H)\rightarrow\mathrm{Aut}(K(H))\rightarrow1$ as done there, but rather the exact sequence $1\rightarrow\{-1,1\}\rightarrow SU(2)\rightarrow\mathrm{Aut}(M_2)\rightarrow1$, i.e. we are restricting to unitaries of determinant $1$, however the same argument applies verbatim \textendash\, this is hinted at, although not quite stated explicitly, in \cite{RaeburnWilliams1998} Hooptedoodle 4.91).  Thus $B_1$ is not isomorphic to the bundle of operators naturally arising from a $U(2)$-bundle where the fibre is a 2-dimensional Hilbert space.  But any C*-algebra of the form $PC(X,M_n)P$, where $P$ is some everywhere rank 2 projection, will be (isomorphic to) the C*-algebra of continuous sections of such a bundle.  As the bundles are not isomorphic, the C*-algebras of sections can not be isomorphic either, by \cite{RaeburnWilliams1998} Hooptedoodle 4.90.
\end{proof}

\begin{qst}
Does \autoref{P1nontriv} also hold for $\mathfrak{P}_n$ where $n\geq2$?
\end{qst}

In order to exhibit another interesting property of $\mathfrak{P}_1$, let us digress a moment to discuss \emph{commutators}, i.e. operators of the form $[a,b]=ab-ba$, for some $a$ and $b$ in a C*-algebra $A$.  Likewise, we call an operator of the form $[a,a^*]$, for some $a\in A$, a \emph{*-commutator}.  For any subset $S\subseteq A$, we let $\mathfrak{c}(S)=\{[a,b]:a,b\in S\}$ and $\mathfrak{c}^*(S)=\{[a,a^*]:a\in S\}$ (so $A_0=\overline{\mathrm{span}}(\mathfrak{c}^*(A))$, by \cite{CuntzPedersen1979} Proposition 2.5).

Note that for self-adjoint $a$ and $b$,
\begin{eqnarray*}
(a+ib)^*(a+ib)-(a+ib)(a+ib)^* &=& 2i(ab-ba)\qquad\textrm{and hence}\\
\mathfrak{c}^*(A) &=& i\mathfrak{c}(A_\mathrm{sa})
\end{eqnarray*}
Furthermore, if $a,b,c,d\in A_\mathrm{sa}$ satisfy $[a+ib,c+id]=[a+ib,c+id]^*$ then we have $[a,c]-[b,d]=0$ so
\begin{eqnarray*}
[a+ib,c+id] &=& i([b,c]+[a,d])\qquad\textrm{and hence}\\
\mathfrak{c}^*(A)\quad\subseteq\quad\mathfrak{c}(A)_{\mathrm{sa}} &\subseteq& \mathfrak{c}^*(A)+\mathfrak{c}^*(A)
\end{eqnarray*}
To see that the first inclusion here can be strict, we need look no further than $\mathfrak{P}_1$.  Indeed, what we have called a *-commutator is usually just referred to in the literature as a self-adjoint commutator.  We avoid this confusing terminology for precisely this reason, leaving the term ``self-adjoint commutator'' to naturally refer merely to a commutator that is self-adjoint, i.e. an element of $\mathfrak{c}(A)_\mathrm{sa}$.

\begin{thm}\label{sacom}
Not all self-adjoint commutators are *-commutators, specifically $\mathfrak{c}^*(\mathfrak{P}_1)\subsetneqq\mathfrak{c}(\mathfrak{P}_1)_\mathrm{sa}$
\end{thm}

\begin{proof}
Let $U,V\in \mathfrak{P}_1$ be such that, for all $z\in\mathbb{C}$,
\begin{eqnarray*}
U(\mathbb{C}(1,z))=\begin{bmatrix}0&\frac{(1,z)}{1+|z|^2}\\ \mathbf{0}&0\end{bmatrix} &\textrm{ and }& V(\mathbb{C}(z,1))=\begin{bmatrix}0&\mathbf{0}\\ \frac{(\overline{z},1)}{1+|z|^2}&0\end{bmatrix},\\
\end{eqnarray*}
As $\mathbb{C}(a,b)=\mathbb{C}(1,b/a)=\mathbb{C}(a/b,1)$, for $a,b\in\mathbb{C}\setminus\{0\}$,
\[UU^*(\mathbb{C}(a,b))=\begin{bmatrix}0&\frac{(1,b/a)}{1+|b/a|^2}\\ \mathbf{0}&0\end{bmatrix}\begin{bmatrix}0&\mathbf{0}\\ \frac{(1,\overline{b/a})}{1+|b/a|^2}&0\end{bmatrix}=\begin{bmatrix}\frac{|a|^2}{|a|^2+|b|^2}&\mathbf{0}\\ \mathbf{0}&0\end{bmatrix},\textrm{ and likewise}\]
\[U^*U(\mathbb{C}(a,b))=\begin{bmatrix}0&\mathbf{0}\\ \mathbf{0}&\frac{|a|^2}{|a|^2+|b|^2}\end{bmatrix},VV^*(\mathbb{C}(a,b))=\begin{bmatrix}0&\mathbf{0}\\ \mathbf{0}&\frac{|b|^2}{|a|^2+|b|^2}\end{bmatrix},\]
\[\textrm{ and }V^*V(\mathbb{C}(a,b))=\begin{bmatrix}\frac{|b|^2}{|a|^2+|b|^2}&\mathbf{0}\\ \mathbf{0}&0\end{bmatrix}.\]
So if $X=U+iV$ and $Y=U^*+iV^*$ then $XY-YX=UU^*-VV^*-U^*U+V^*V=\begin{bmatrix}1&\mathbf{0}\\ \mathbf{0}&-1\end{bmatrix}$.  However, we know this is not a *-commutator by \autoref{*com}.
\end{proof}

It would be interesting to know if \autoref{sacom} holds for a more elementary kind of C*-algebra.  Indeed, it seems plausible that in $C(S^2,M_2)$ the difference of the Bott projection and its orthogonal complement could be another self-adjoint commutator that is not a *-commutator.  It would also be nice to know if the inclusion $\mathfrak{c}(A)_{\mathrm{sa}}\subseteq\mathfrak{c}^*(A)+\mathfrak{c}^*(A)$ can be strict for some C*-algebra $A$.  Indeed, if we recursively define $\mathfrak{c}_{n+1}^{(*)}(A)=\mathfrak{c}_n^{(*)}(A)+\mathfrak{c}^{(*)}(A)$, the following question naturally arises.

\begin{qst}
For any $n\in\mathbb{N}$, can we find C*-algebras $A$ such that either or both inclusions in $\mathfrak{c}^*_n(A)\subseteq\mathfrak{c}_n(A)_{\mathrm{sa}}\subseteq\mathfrak{c}^*_{2n}(A)$ are strict?
\end{qst}

\bibliography{maths}{}
\bibliographystyle{plainurl}

\end{document}